\documentclass[12pt]{amsart}
\usepackage[foot]{amsaddr}
\usepackage{amsrefs}
\usepackage[english]{babel}
\usepackage[utf8]{inputenc}
\usepackage{amsmath,relsize,hyperref}
\usepackage{graphicx}
\usepackage{amssymb}
\usepackage{amsthm}
\usepackage{float}

\numberwithin{equation}{section}

\usepackage{tikz-cd}
\tikzset{every loop/.style={}}
\tikzset{
  symbol/.style={
    draw=none,
    every to/.append style={
      edge node={node [sloped, allow upside down, auto=false]{$#1$}}}
  }
}

\usepackage{mathrsfs}

\usepackage{enumitem}
\usepackage{yfonts}
\usepackage{ dsfont }

\usetikzlibrary{matrix}
\usetikzlibrary{calc,intersections}
\usepackage{adjustbox}

\newtheorem{thm}{Theorem}[section]
\newtheorem{lem}[thm]{Lemma}
\newtheorem{prop}[thm]{Proposition}

\newtheorem{cor}[thm]{Corollary}

\theoremstyle{definition}
\newtheorem{eg}[thm]{Example}

\newtheorem{remark}[thm]{Remark}

\newcommand{\E}{\mathbb{E}}

\renewcommand{\O}{\mathcal{O}}

\renewcommand{\H}{{\mathcal H}}
\newcommand{\M}{{\mathcal M}}
\newcommand{\ie}{\emph{i.e.,} }

\newcommand{\Mbar}{\overline{{\mathcal M}}}

\newcommand{\Z}{\mathbb{Z}}
\newcommand{\Q}{\mathbb{Q}}
\newcommand{\C}{\mathbb{C}}

\newcommand{\Aff}{\mathbb{A}}
\newcommand{\Pro}{\mathbb{P}}

\newcommand{\G}{\mathbb{G}}
\DeclareMathOperator{\Hom}{Hom}

\DeclareMathOperator{\rank}{rank}

\DeclareMathOperator{\Spec}{Spec}

\DeclareMathOperator{\Sym}{Sym}
\DeclareMathOperator{\Proj}{Proj}

\DeclareMathOperator{\GL}{GL}

\DeclareMathOperator{\Chow}{CH}
\DeclareMathOperator{\K}{K}
\DeclareMathOperator{\GG}{G}

\newcommand{\coloneq}{:=}

\DeclareMathOperator{\Ind}{Ind}
\DeclareMathOperator{\Res}{Res}

\usepackage[margin=1in]{geometry}

\usepackage{microtype}
\begin{document}

\title{The $\K$-theory of the moduli stacks $\M_2$ and $\Mbar_2$}

\author{Dan Edidin}
\address{Department of Mathematics, University of Missouri, Columbia, MO 65211}
\email{edidind@missouri.edu}

\author{Zhengning Hu}
\address{Department of Mathematics, University of Arizona, Tucson, AZ 85721}
\email{zhengninghu@math.arizona.edu}

\begin{abstract}
We compute the integral Grothendieck rings of the moduli stacks, $\M_2$, $\Mbar_2$ of
smooth and stable curves of genus two respectively.
We compute $\K_0(\M_2)$ by using the
presentation of $\M_2$ as a global quotient stack given by Vistoli
\cite{Vis:98}. To compute the Grothendieck ring $\K_0(\Mbar_2)$ we
decompose $\Mbar_2$ as $\Delta_1$ and its complement $\Mbar_2 \setminus \Delta_1$ and use their presentations as quotient stacks given by Larson \cite{Lar:21}
to compute the Grothendieck rings. We show that they are torsion-free
and this, together with the Riemann-Roch isomorphism allows us to ultimately
give a presentation for the integral Grothendieck ring $\K_0(\Mbar_2)$.
\end{abstract}

\maketitle

\section{Introduction}

Let $\M_g$ be the moduli stack of smooth curves of genus $g \geq 2$
over a fixed ground field $k$, and let $\Mbar_g$ be its Deligne-Mumford
compactification.  In his landmark paper, Mumford \cite{Mum:83} 
showed that the Chow groups of the coarse moduli space
$\overline{M}_g$ admit a $\Q$-valued intersection product and gave a
full description of the Chow group $\Chow^*(\overline{M}_2)_\Q$ together with
the multiplication table for divisor classes.  After the development of
equivariant intersection theory \cite{EdGr:98}, 
equivariant techniques were used to compute the integral Chow rings of the moduli
stacks $\Mbar_{1,1}, \Mbar_2, \Mbar_{2,1}, \Mbar_3$ \cite{EdGr:98, Lar:21,ViDL:21,dilorenzo2021stable, pernice2023almost} as well
as the stack $\H_g$ parametrizing smooth hyperelliptic curves \cite{EdFu:09,dilorenzo2019chow}.
In this paper, we turn our attention to $\K$-theory and
use equivariant methods to compute the Grothendieck ring of the
Deligne-Mumford stacks $\M_2$ and $\Mbar_2$.

In \cite{Vis:98} Vistoli computed the integral Chow ring
$\Chow^*(\M_2)$ by identifying the stack $\M_2$ with the global
quotient stack $[(V \setminus D)/\GL_2]$ where $V$ is a representation
of $\GL_2$ and $D \subset V$ is a $\GL_2$-invariant subscheme of $V$.
The first result of this paper, is to use Vistoli's 
presentation for $\M_2$
to compute $\K_0(\M_2)$.

\begin{thm}\label{thm-M2}
The Grothendieck ring of the moduli stack $\M_2$ is given by
\[\K_0(\M_2) = \dfrac{\Z[\epsilon,\lambda^{\pm 1}]}{(h_0 + (\epsilon\lambda^{-1})h_3 - \lambda^{-2}h_8,h_1 + \epsilon h_2 - \lambda^{-1}h_7)},\]
where $\epsilon$ is the class of the Hodge bundle $\E$ on $\M_2$, $\lambda$
and $\lambda^{-1}$ are respectively the classes of $\det \E$ and $\det \E^\vee$, and $h_n$ is the class of the symmetric power $\Sym^n \E$. In particular,
it is a free abelian group of rank 18 with a basis given in Appendix~\ref{app.Zbasis}.
\end{thm}

It is not known whether $\Mbar_2$ can be expressed as a quotient
stack $[(V\setminus Y)/G]$ where $V$ is a representation of a linear
algebraic group $G$ with $Y \subset V$ a $G$-invariant subscheme of
$Y$.  (Following the terminology of \cite{ViDL:21} we call such a
presentation an {\em easy presentation}.) As a result, the computation
of $\Chow^*(\Mbar_2)$ is much more difficult than the corresponding
computation for the stack $\M_2$.  There are two independent
computations of $\Chow^*(\Mbar_2)$.  In \cite{Lar:21}, Larson stratified
$\Mbar_2$ into the boundary divisor $\Delta_1$ and its complement $\Mbar_2 \setminus \Delta_1$ and showed that both strata have easy presentations.
The
ring structure on $\Chow^*(\Mbar_2)$ is computed using the
localization exact sequence for higher Chow groups and a 
careful study of the first higher Chow groups of $\Mbar_2\setminus \Delta_1$.
In an alternative
approach that avoids the use of higher Chow groups, Di Lorenzo and  Vistoli
\cite{ViDL:21} embedded $\Mbar_2$ as an open substack of a vector bundle
over the stack $\mathcal{P}$ of polarized twisted conics.  The
integral Chow ring $\Chow^*(\mathcal{P})$ is computed using a patching
technique for integral Chow rings of Artin stacks, and the Chow ring
of the Deligne-Mumford stack $\Mbar_2$ is realized as quotient of this ring.

The approach used here is closer to Larson's approach for computing
$\Chow^*(\Mbar_2)$. However, we can avoid the use of higher $\K$-theory as we now explain.
We first use equivariant $\K$-theory to compute 
the Grothendieck rings of $\Delta_1$ and its complement $\Mbar_2\setminus\Delta_1$,
by viewing them as stack quotients in open sets in representations \cite{Lar:21}.
However, unlike for Chow groups, we can show that the Grothendieck groups $\K_0(\Delta_1)$ and $\K_0(\Mbar_2\setminus \Delta_1)$
are torsion-free of ranks $65$ and $32$ respectively. It follows
from the localization sequence
\[ \K_0(\Delta_1) \to \K_0(\Mbar_2) \to \K_0(\Mbar_2\setminus\Delta_1) \to 0\]
that $\K_0(\Mbar_2)$ has rank at most $97$. 

On the other hand, we can compute the rank of 
$\K_0(\Mbar_2)$ using 
the Riemann-Roch theorem for Deligne-Mumford stacks \cite{Edi:13}
which states that there is a Riemann-Roch isomorphism $\K_0(\Mbar_2) \otimes \C
\stackrel{\tau} \to \Chow^*(I\Mbar_2)_\C$, where $I\Mbar_2$ is the inertia stack.
The rank of $\Chow(I\Mbar_2)$ can be computed using the results in Pagani's
thesis \cite{Pag:09} for Chen-Ruan cohomology, although it would not be too difficult to make the calculation directly, and we conclude that $\rank \K_0(\Mbar_2) = 97$. 
It follows that we have a short exact sequence
\[ 0 \to  \K_0(\Delta_1) \to \K_0(\Mbar_2) \to \K_0(\Mbar_2\setminus\Delta_1) \to 0\]
and thus $\K_0(\Mbar_2)$ is free abelian of rank $97$.
Using this exact sequence and our descriptions of the generators for
$\K_0(\Delta_1)$ and $\K_0(\Mbar_2\setminus \Delta_1)$, we obtain the following description of $\K_0(\Mbar_2)$.


\begin{thm}\label{thm-Mbar2}
The Grothendieck group $\K_0(\Mbar_2)$ is freely generated as an abelian group of rank $97$. Moreover, as a ring, 
\[\K_0(\Mbar_2) \cong \Z[\epsilon,\lambda^{\pm 1},\delta_1^{\pm 1}]/
\left((I_{\Delta_1} \cap I_{\Mbar_2 \setminus \Delta_1}) + \left(1-\delta_1^{-1}\right)I_{\Delta_1}\right).\]
Here $\epsilon$ is the class of Hodge bundle $\E$ on $\Mbar_2$, $\lambda$ is the class of $\det \E$, and $\delta_1$ is the class of the boundary divisor $\Delta_1$. The ideal $I_{\Delta_1}$ is generated by five polynomials in
$\epsilon, \lambda^{\pm 1}, \delta_1^{\pm 1}$ and
the ideal $I_{\Mbar_2 \setminus \Delta_1}$ is generated by three polynomials
  in $\epsilon, \lambda^{\pm 1}$.
\end{thm}

\subsection{Further directions}
Our proof takes advantage that we can stratify $\Mbar_2$ as a divisor $\Delta_1$
and its complement $\Mbar_2 \setminus \Delta_1$ where both strata have torsion-free 
Grothendieck groups. We suspect that this may be a general phenomenon
whereby open sets in representations of reductive groups have torsion-free
equivariant Grothendieck groups. The intuitive reason for this expectation is that torsion relations
in the Chow group are ``exponentiated'' when they appear in the Grothendieck group.
For example if $D$ is a Cartier divisor on a variety $X$, the relation
$n[D] = 0 \in \Chow^1(X)$ corresponds to the relation $\mathcal{O}(D)^n =1$
in $\K_0(X)$.

If this is a case, then there is reason to believe that the integral Grothendieck rings of other moduli spaces of curves of low genus such $\Mbar_{2,1}$ and $\Mbar_{3}$ can be computed using a relatively direct stratification, and that these computations will be easier than those made for Chow groups in
\cite{dilorenzo2021stable} and \cite{pernice2023almost}.

\begin{remark}[Remark on the ground field] We work over an algebraically closed field. For the calculations of
  $\K_0(\M_2), \K_0(\Delta_1), \K_0(\Mbar_2 \setminus \Delta_1)$ we can work over a field whose characteristic is not 2 or 3. However, for the computation of
  $\K_0(\Mbar_2)$ we use the Riemann-Roch theorem of \cite{Edi:13} which is only proved over $\C$.
\end{remark}

\section{Background on equivariant $\K$-theory} \label{sec.Ktheory}
We give a brief background on equivariant $\K$-theory. For more details
see \cite{Tho:87}.  Let $G$ be an algebraic group acting on a
separated scheme $X$ of finite type over a field $k$.
We denote by $\GG_0(X,G)$ the Grothendieck group of $G$-equivariant
coherent sheaves on $X$ and $\K_0(X,G)$ the Grothendieck group of
$G$-equivariant vector bundles. Tensor product of bundles makes
$\K_0(X,G)$ into a ring which we call the $G$-equivariant Grothendieck ring
of $X$. The Grothendieck group $\GG_0(G,X)$ is a module for this ring.
When $X$ is regular then Thomason's resolution theorem \cite{Tho:87}
implies that
the natural map $\K_0(X,G) \to \GG_0(X,G)$ is an isomorphism. In particular,
if $X = \Spec k$ then $\K_0(X,G) = \GG_0(X,G) = R(G)$ where
$R(G)$ is the (modular) representation ring of $G$.
If $G$ acts freely on $X$ with quotient scheme $X/G$, then
we can identify $\GG_0(X,G)$ with
$\GG_0(X/G)$ and $\K_0(X,G)$ with $\K_0(X/G)$.

If $G$ is connected reductive with simply connected
commutator subgroup, for example $\GL_n$, and $T \subset G$ is a split
maximal torus of $G$, for example $T_n \subset \GL_n$, then the
restriction homomorphism of modules $\GG_0(X,G) \to \GG_0(X,T)$
induces the following relation between $T$-equivariant and $G$-equivariant
$\K$-theory.

\begin{prop}\cite[Proposition 31]{Mer:05}\label{prop.changegroup}
Under the above assumption on $G$ and $T$, the restriction induces an $R(T)$-module isomorphism
\[\GG_0(X,G) \otimes_{R(G)} R(T) \longrightarrow \GG_0(X,T).\]
\end{prop}


The localization sequence for equivariant Grothendieck groups states that:
\begin{prop}\cite[Theorem 2.7]{Tho:87}\label{prop.longexact}
Consider a closed embedding $i : Z \hookrightarrow X$ with the complementary open immersion $j : U = X \setminus Z \hookrightarrow X$. There is an exact sequence
\[ \GG_0(Z,G) \stackrel{i_*}{\to} \GG_0(X,G) \stackrel{j^*}{\to} \GG_0(U,G) \to 0,\]
which can be extended on the left with the higher $G$-equivariant $\K$-theory groups.
\end{prop}

Moreover, if $i : Z \hookrightarrow X$ is a closed embedding of
regular $G$-schemes, then we have the following equivariant self-intersection formula, which is proved in the course of the proof of \cite[Theorem 3.7]{VeVi:02}.

\begin{prop}\cite{VeVi:02}\label{prop.selfintersection}
For any $\alpha \in \K_0(Z,G)$, 
\[i^*i_*(\alpha) = \lambda_{-1}\left(N_i^\vee\right) \cap \alpha\]
where $N_i^\vee$ is the conormal bundle associated to the embedding $i$, and the $\K$-theoretic Euler class $\lambda_{-1}\left(N_i^\vee\right)$ is defined by $\sum_{n = 0}^{\rank N_i^\vee} (-1)^n [\wedge^n N_i^\vee]$.
\end{prop}

Another fundamental property of equivariant $\K$-theory we use in the computation is the localization theorem for action of diagonalizable groups, which allows us to reduce the global calculations to those on the fixed point locus.

\begin{prop}\cite[Theorem 3.2]{Nie:74}\label{prop.localization}
Let $D$ be a diagonalizable group acting on a smooth quasi-projective scheme $X$ thus the fixed point scheme $X^D$ is smooth. Let $i : X^D \hookrightarrow X$ be the inclusion of the fixed locus of $D$ in $X$, and $S$ the multiplicative subset of $R(D)$ generated by $\K$-theoretic Euler classes of nontrivial characters of $D$. Then the natural $R(D)$-module homomorphism
\[i_* : S^{-1}\K_0(X^D,D) \longrightarrow S^{-1}\K_0(X,D)\] 
is an isomorphism, with an explicit inverse given by multiplication by $\lambda_{-1}\left(N_{X^D}^\vee X\right)$ arising from the self-intersection formula for the regular embedding $i : X^D \hookrightarrow X$. Namely, for $\alpha \in S^{-1}\K_0(X,D)$, we have
\[\alpha = i_*\left(\dfrac{i^* \alpha}{\lambda_{-1}\left(N_{X^D}^\vee X\right)}\right).\]
\end{prop} 


We also use the projective bundle theorem for equivariant $\K$-theory extensively.

\begin{prop}\label{prop.projectivebundle}
  Let $X$ be a regular scheme and let $V$ be a $G$-equivariant vector
  bundle on $X$. Then there is an isomorphism of Grothendieck rings
  $\K_0(\Pro(V),G) = \frac{\K_0(X,G)[t]}{ \lambda_{-1}(V \otimes t)^*}$
  where $t$ is the class
  of ${\mathcal O}_{\Pro(V)}(1)$.
\end{prop}

Finally, if $G$ acts properly on a smooth scheme $X$
and we work over $\C$, we have the following
Riemann-Roch isomorphism \cite{EdGr:05, Edi:13}.
\begin{thm} \label{thm.rr}
  There is a Riemann-Roch isomorphism $\tau_X \colon \K_0(X,G)_\C \to
  \Chow^*(I[X/G])_\C$. In particular $\K_0(X,G)_\Q$ and $\Chow^*(I[X/G])_\Q$
  have the same rank.
  (Here $I[X/G]$ is the inertia stack of the quotient stack.)
\end{thm}

\section{Torus equivariant $\K$-theory of projective spaces of binary forms}
In this section, we prove some results and propositions essential to the final computations.

Let $V$ be the defining representation of $\GL_2$ of dimension $2$ and $\mathcal{D} : \GL_2 \to \G_m$ its determinant. Following the notation used in
\cite{Lar:21} set
\[V_N = \Sym^N V^*,\]
and let $\Pro^N = \Pro(V_N)$. We denote by $T_2 = \G_m \times \G_m$ the maximal torus of $\GL_2$ with induced action on $\Pro^1 = \Pro(V^*)$ given by
\[t \cdot (x_0,x_1) = (\chi_1^{-1}(t) x_0 : \chi_2^{-1}(t) x_1).\]
The induced $T_2$-action on $\Pro^N$ is given by 
\[t \cdot (X_0 : \cdots : X_i : \cdots : X_N) = \left(\chi_1^{-N}(t)X_0 : \cdots : \chi_1^{-N + i}(t)\chi_2^{-i}(t)X_i : \cdots : \chi_2^{-N}(t)X_N\right)\]
where we denote by $X_i$ the coefficient of the monomial $x_0^{N - i}x_1^i$. 


With the above convention, we are considering 
\[\Pro^N = \Proj k[x_0,\cdots,x_N]\]
where $x_0,\cdots,x_N$ correspond to global sections of
$\O_{\Pro^N}(1)$. As a consequence, the action of $T_2$ on global sections
of $\mathcal{O}(1)$ is dual to the action on the coordinates of $\Pro^N$.

\subsection{$T_2$-equivariant class $[\O_H]_{T_2}$ of an invariant hypersurface $H$}\label{subsec.hypersurface}

Let $V$ be a finite dimensional representation of a split torus $T$ over $k$. Let $H$ be a $T$-invariant hypersurface of $\Pro(V)$ defined by a homogeneous form $f \in \Sym^d V^*$ with $\mu \cdot f = \chi^{-1}(\mu) f$ for some character $\chi : T \to \G_m$ of $T$. We denote by $t = [\mathcal{O}_{\Pro(V)}(1)]$ the $T$-equivariant $\K$-theoretic hyperplane class of $\Pro(V)$. By the short exact sequence from the closed embedding $H \hookrightarrow \Pro(V)$,
\[0 \to \mathcal{O}_{\chi^{-1}}(-d) \to \mathcal{O}_{\Pro(V)} \to \mathcal{O}_H \to 0\]
we obtain the $T$-equivariant $\K$-theoretic class of $H$ as below
\begin{align*}
[\mathcal{O}_H]_{T} &= [\mathcal{O}_{\Pro(V)}] - [\mathcal{O}_{\chi^{-1}}(-d)] \\
&= 1 -  \chi^{-1}t^{-d}.
\end{align*}


\begin{eg}\label{eg.prod}
  Now let $V$ be the standard $2$-dimensional representation of $T_2$
  with decomposition $V = L_1 \oplus L_2$ into $\chi_1$ and $\chi_2$-eigenspaces. Consider $\Pro^2 = \Pro(\Sym^2 V^*)$ where $T_2$ acts with
  weights $\chi_1^{-2}, \chi_1^{-1}\chi_2^{-1}, \chi_2^{-2}$ respectively.
  The fixed locus of $T_2$ in $\Pro^2$ consists of
  the three points $Q_1 = (1:0:0)$, $Q_2 = (0:1:0)$ and $Q_3 = (0:0:1)$. Each fixed point is defined by the complete
  intersection of two hyperplanes and we can compute their $\K$-theoretic
  classes using the formula for the $\K$-theoretic class of a hyperplane.
For example   
\[Q_1 = H_{\chi_1\chi_2} \cap H_{\chi_2^2},\]
and we obtain
\[[Q_1]_{T_2} = (1 - \chi_1\chi_2t^{-1})(1 - \chi_2^2t^{-1})\]
where $t = [\mathcal{O}_{\Pro^2}(1)]$ denotes the hyperplane class.
Likewise,
\[[Q_2]_{T_2} = (1 - \chi_1^2t^{-1})(1-\chi_2^2 t^{-1}), \quad [Q_3]_{T_2} = (1-\chi_1^2t^{-1})(1-\chi_1\chi_2t^{-1}).\]
\end{eg}


\subsection{Pushforward along the multiplication map for binary
forms}\label{subsec.pushforward}
Let $V$ be the standard two-dimensional representation of $\GL_2$ and
consider the map $\Sym^r V^* \times \Sym^{N-2r}V^* \to \Sym^N V^*$
which sends $(f,g) \mapsto f^2g$.
There is an induced map of projective spaces
$$\pi_r \colon \Pro^r \times \Pro^{N-2r} \longrightarrow \Pro^{N}.$$ 
We describe how to use
the localization theorem in equivariant $\K$-theory to describe
the image of the pushforward
$$(\pi_r)_* \colon \K_0(\Pro^r\times \Pro^{N-2r}, T_2) \to \K_0(\Pro^N,T_2)$$
where $T_2 \subset \GL_2$ is the maximal torus of diagonal matrices.

By the projective bundle theorem Proposition \ref{prop.projectivebundle},
$T$-equivariant $\K$-theory of $\Pro^s$ can be identified as 
$$\K_0(\Pro^s,T_2) = \dfrac{R(T)[\mathcal{O}_{\Pro^s}(1)]}
  {\left(\prod_{k = 0}^s (1 - a^{s-k}b^k \mathcal{O}_{\Pro^s}(1)^{-1})\right)},$$
where $a,b$ denote the $\K$-theoretic classes of the characters of
the maximal torus $T_2$.
Let $x_1 = [\O_{\Pro^r}(1)]$, $x_2 = [\O_{\Pro^{N - 2r}}(1)]$ and $t =
[\O_{\Pro^N}(1)]$. Since
$\pi_r^*(t) = x_1^2x_2$
the pushforward $(\pi_r)_*(\K_0(\Pro^r \times \Pro^{N - 2r},T_2))$ is generated
as an $\K_0(\Pro^N,T_2)$-module
by the classes
$$(\pi_r)_*(1),(\pi_r)_*(x_1),\cdots,(\pi_r)_*(x_1^r).$$


Consider the following commutative diagram of projective spaces and their fixed loci.
\[\begin{tikzcd}
(\Pro^r \times \Pro^{N - 2r})^{T_2} \ar[r,"i",hook] \ar[d,"\pi_r^{T_2}"] & \Pro^r \times \Pro^{N - 2r} \ar[d,"\pi_r"] \\
(\Pro^N)^{T_2} \ar[r,"j",hook] & \Pro^N
\end{tikzcd}\]
By localization theorem Proposition \ref{prop.localization}, the induced pushforwards $i_*$ and $j_*$ are both isomorphisms after localizing at the multiplicative set $S$ consisting of the 
Euler classes of the non-zero characters of $T_2$.
Since $\K_0(\Pro^n,T_2)$ is a free $R(T_2)$-module we can 
compute the pushforwards of the line elements $x_1^k$
by computing $j_*(\pi_r^{T_2})_*\alpha_k$ where $\alpha_k \in S^{-1}\K_0\left((
\Pro^r \times \Pro^{N-2r})^{T_2}\right)$ satisfies $x_1^k = i_* \alpha_k$.

By the self-intersection formula,
\begin{eqnarray*}
  \alpha_k & = & \dfrac{i^*x_1^k}{\lambda_{-1}\left(N^\vee_{(\Pro^r \times \Pro^{N - 2r})^T} \Pro^r \times \Pro^{N - 2r}\right)}\\
  & = & \dfrac{[\O(k,0)\mid_{(\Pro^r \times \Pro^{N - 2r})^T}]}{\lambda_{-1}\left(N^\vee_{(\Pro^r \times \Pro^{N - 2r})^T} \Pro^r \times \Pro^{N - 2r}\right)}.
\end{eqnarray*}
The fixed-point locus of $\Pro^r \times \Pro^{N -2r}$ is a finite set consisting of the points
\[P_{i,j} = ((0 : \cdots : 1 : \cdots : 0),(0 : \cdots : 1 : \cdots : 0))\]
where the only nonzero entry is located on the $i$-th component for the first factor and on the $j$-th component for the second factor of the product of projective spaces $\Pro^r \times \Pro^{N - 2r}$.
The normal bundle at $P_{i,j}$ can be directly computed as
\begin{align*}
N_{P_{i,j}} \left(\Pro^r \times \Pro^{N - 2r}\right) &= \Hom\left(L_i,k^{r + 1}/L_i\right) \oplus \Hom\left(L_j,k^{N - 2r + 1}/L_j\right) \\
&= \left(\bigoplus_{l \ne i}L_i^\vee \otimes L_l\right) \oplus \left(\bigoplus_{l \ne j} L_j^\vee \otimes L_l\right)
\end{align*}
where $L_l$ is the $l$-th coordinate line in the representation
$V_r$. We can compute its $\K$-theoretic Euler class as
\[\lambda_{-1}\left(N^\vee_{P_{i,j}}\left(\Pro^r \times \Pro^{N-2r}\right)\right) = \left[\prod_{k \neq i} \left(1 - \dfrac{a^{r - k}b^k}{a^{r - i}b^i}\right)\right] \left[\prod_{k \ne j}\left(1 - \dfrac{a^{N - 2r - k}b^k}{a^{N - 2r - j}b^j}\right)\right]\]
and thus in localized $\K$-theory
\[x_1^k = i_*\left( \sum_{(i,j)}\dfrac{(a^{r - i}b^i)^k}{\left[\prod_{k \neq i} \left(1 - \dfrac{a^{r - k}b^k}{a^{r - i}b^i}\right)\right] \left[\prod_{k \ne j}\left(1 - \dfrac{a^{N - 2r - k}b^k}{a^{N - 2r - j}b^j}\right)\right]}\right).\]

For the restriction map $\pi_r^{T_2}$ between finite fixed-point loci, it is easy to see that
\[\pi_r^{T_2}(P_{i,j}) = Q_{2i+j} \in (\Pro^N)^{T_2},\]
where $Q_{2i + j}$ denotes the point in $\Pro^N$ with the only nonzero entry
located in the $(2i + j)$-th coordinate.
Since the equivariant $\K$-theory of the finite set $(\Pro^N)^{T_2}$ is generated by the classes $[Q_k]_{T_2}$ as an $R(T_2)$-module, the pushforward $j_*$ on the level of Grothendieck rings is determined by the image of generators $[Q_k]_{T_2}$, which is
\[j_*([Q_k]_{T_2}) = \prod_{i \neq k}\left(1 - a^{N - i}b^it^{-1}\right)\]
where $t = [\O_{\Pro^N}(1)]$ is the hyperplane class. Therefore we can explicitly compute the pushforward of $1,x_1,\cdots,x_1^r$ along the multiplication map $\pi_r$. 

\begin{eg}\label{eg.Veronese}
Consider the the simplest example where $N=2$ and $r=1$. In this case we are considering
the Veronese map $\pi_1 : \Pro^1 \to \Pro^2$ sending $(x_0 : x_1)$ to $(x_0^2 : x_0x_1 : x_1^2)$. From Section \ref{subsec.pushforward}, $\K_0(\Pro^1,T_2)$ is generated additively by $1$ and $x = [\mathcal{O}_{\Pro^1}(1)]$ as an $R(T_2)$-module. We want to find the pushforwards $(\pi_1)_*(1),(\pi_1)_*(x)$ in $\K_0(\Pro^2,T_2)$. 

In this case we have the following commutative diagram
\begin{equation}\label{diag.veronese}
\begin{tikzcd}
(\Pro^1)^{T_2} \ar[r,"i",hook] \ar[d,"\pi_1^{T_2}"] & \Pro^1 \ar[d,"\pi_1"] \\
(\Pro^2)^{T_2} \ar[r,"j",hook] & \Pro^2
\end{tikzcd}
\end{equation} 
where we denote by $(\Pro^1)^{T_2} = \{P_0 = (0:1),P_\infty = (1:0)\}$ and $(\Pro^2)^{T_2} = \{Q_1 = (1:0:0),Q_2 = (0:1:0),Q_3 = (0:0:1)\}$.
By the self-intersection formula, Proposition \ref{prop.selfintersection}, 
we know that for any class $E \in \K_0(\Pro^1,T_2)$, $i_*i^*E$ and $E$ differ by the Koszul factor 
which is given by the $\K$-theoretic Euler class of the conormal bundle of the fixed locus $(\Pro^1)^{T_2}$ in $\Pro^1$.
Since
\[[N_{P_0}\Pro^1] = a^{-1}/b^{-1},\quad [N_{P_\infty}\Pro^1] = b^{-1}/a^{-1},\]
we have
\begin{align*}
& i^*[\O_{\Pro^1}] = \left(\dfrac{[\O]}{1 - a/b},\dfrac{[\O]}{1 - b/a}\right) = \left(\dfrac{b}{b - a},\dfrac{-a}{b - a}\right), \\
& i^*[\O_{\Pro^1}(1)] = \left(\dfrac{[\O(1)]}{1 - a/b},\dfrac{[\O(1)]}{1 - b/a}\right) = \left(\dfrac{b^2}{b - a},\dfrac{-a^2}{b - a}\right).
\end{align*}

Notice that along the restriction map $\pi_1^{T_2}$ on the fixed loci, we have $\pi_1^{T_2}(P_0) = Q_3$ and $\pi_1^{T_2}(P_\infty) = Q_1$. In Example \ref{eg.prod} from Section \ref{subsec.hypersurface} we computed that 
\begin{align*}
& [Q_1]_{T_2} = (1 - abt^{-1})(1 - b^2t^{-1}), \\
& [Q_3]_{T_2} = (1 - a^2t^{-1})(1 - abt^{-1}),
\end{align*}
where $t = [\O_{\Pro^2}(1)]$. Using the localization theorem
we deduce from the commutative diagram \eqref{diag.veronese} that
\begin{align*}
(\pi_1)_*(1) &= \dfrac{b}{b - a}[Q_3]_{T_2} - \dfrac{a}{b - a}[Q_1]_{T_2} \\
&= (1 - abt^{-1})(1 + abt^{-1}),
\end{align*}
and
\begin{align*}
(\pi_1)_*(x) &= \dfrac{b^2}{b - a}[Q_3]_{T_2} - \dfrac{a^2}{b - a}[Q_1]_{T_2} \\
&= (a + b)(1 - abt^{-1}).
\end{align*}
 \end{eg}

\section{Proof of Theorem \ref{thm-M2}. The Grothendieck ring of the moduli stack $\M_2$}

First recall the result from \cite{Vis:98} that the moduli stack $\M_2$ can be described as a global quotient stack of an open subset of a representation of $\GL_2$. Let $\Aff(1,6)$ be the affine space of binary forms of degree $6$ and $\Delta' \subset \Aff(1,6)$ its discriminant locus, consisting of forms of multiple roots.  

\begin{prop}\cite[Proposition 3.1]{Vis:98}
Denote by $\Aff_{sm}(1,6) \coloneq \Aff(1,6)\setminus\Delta'$ the open subset of $\Aff(1,6) \cong \Aff^7$ consisting of smooth nonzero binary forms of degree $6$ with distinct roots. Then $\M_2$ is canonically isomorphic to the quotient $[\Aff_{sm}(1,6)/\GL_2]$ with action of $\GL_2$ on $\Aff(1,6)$ given by $A \cdot f(x) = \det(A)^2 f(A^{-1}x)$ for $A \in \GL_2$.
\end{prop}

Since $\Aff_{sm}(1,6)$ is an open subset of a $\GL_2$-representation,
the equivariant Grothendieck ring $\K_0(\Aff_{sm}(1,6),\GL_2)$ is a quotient
of the representation ring of $\GL_2$ subject to some relations.
To establish notation we identify the representation
ring $R(\GL_2) = \Z[e_1, e_2^{\pm 1}]$ where $e_1$ is the class of defining representation $V$,
and $e_2 = [\det V]$. Note that the class of dual representation
$V^*$ equals $e_1 \otimes e_2^{-1}$.

Following
the strategy used by Vistoli to compute the Chow ring
of $\M_2$ we first pass to the projectivization to eliminate the global scaling.
Let $\Delta$ be the projectivization of the discriminant locus $\Delta'$
in $\Pro(\Aff(1,6)) \cong \Pro^6$.
The quotient map $\Aff_{sm}(1,6) \to \Pro^6 \setminus \Delta$ induces
an 
epimorphism of Grothendieck rings
\[\K_0(\Pro^6 \setminus \Delta,\GL_2) \longrightarrow \K_0(\Aff_{sm}(1,6),\GL_2)\]
with kernel generated by the $\K$-theoretic Euler class of
line bundle ${\mathcal O}(-1) \otimes \mathcal{D}^{\otimes 2}$. 
If we denote by $t$ the class of ${\mathcal O}(1)$ in $\K_0(\Pro^6, \GL_2)$ then 
\[\K_0(\M_2) = \dfrac{\K_0(\Pro^6 \setminus \Delta,\GL_2)}{(1 - te_2^{-2})}.\]

By the localization sequence 
\[\GG_0(\Delta,\GL_2) \to \K_0(\Pro^6,\GL_2) \to \K_0(\Pro^6 \setminus \Delta,\GL_2) \to 0\]
it suffices to find the ideal in $\K_0(\Pro^6, \GL_2)$ generated by
the image of $\GG_0(\Delta, \GL_2)$. Following
the strategy used in \cite{EdFu:09} we use the fact
that $\K_0(\Pro^6, \GL_2)$ is a summand in $\K_0(\Pro^6, T_2)$
and first compute the ideal generated by the image of $\GG_0(\Delta,T_2)$
in $\K_0(\Pro^6, T_2)$.
Here we denote by $T_2$ the maximal torus in $\GL_2$ of diagonal matrices.

\subsection{Computation of the image of $\GG_0(\Delta, T_2)$ in $\K_0(\Pro^6, T_2)$} \label{sec.chowenv}
Denote by $\Delta = \Delta_1,\Delta_2,\Delta_3$ the closed subsets of $\Pro^6$ corresponding to sextic binary forms divisible by the square of a form of degree $1,2$ and $3$ respectively over some extension of the base field $k$.
As observed in \cite{Vis:98}, $\Delta_r$ is the image
of the map $\pi_r \colon \Pro^r \times \Pro^{6-2r} \to \Pro^6$
given by $(f,g) \mapsto f^2g$.
\begin{lem}
  The image of $\GG_0(\Delta, T_2)$ in $\K_0(\Pro^6, T_2)$ is 
  the ideal generated by images of the pushforwards
\[(\pi_r)_* : \K_0(\Pro^r \times \Pro^{6 - 2r},T_2) \longrightarrow \K_0(\Pro^6,T_2).\]
\end{lem}
\begin{proof}
It has been shown in \cites{Vis:98,EdFu:09} that
\[\pi : \coprod_{r = 1}^3 \Pro^r \times \Pro^{6 - 2r} \longrightarrow \Delta_1\]
is an equivariant Chow envelope in the sense
of \cite{EdGr:98} meaning that 
for every $T_2$-invariant subvariety $V \subset \Delta_1$
there is an invariant subvariety $V'$ of $\coprod_{r=1}^3 \Pro^r \times \Pro^{6-2r}$ mapping birationally onto $V$. Hence by \cite[Proposition 2.2]{AnPa:15}
the pushforward
$$\pi_* \colon \coprod_{r = 1}^3 \K_0(\Pro^r \times \Pro^{6 - 2r},T_2) \longrightarrow \K_0(\Delta_1,T_2)$$
is surjective and the lemma follows.
\end{proof}

Now let $t = [\O_{\Pro^6}(1)]$, and for $r = 1,2,3$, let
\[x_{r,1} = [\O_{\Pro^r}(1)], \quad x_{6-2r,2} = [\O_{\Pro^{6 - 2r}}(1)].\]
Then $x_{r,1},x_{6-2r,2}$ and $t$ satisfy the relation
\[\pi_r^*(t) = x_{r,1}^2x_{6 - 2r,2}.\]
Therefore for each $r = 1,2,3$, the pushforward
$(\pi_r)_*\K_0(\Pro^r \times \Pro^{6-2r}, T_2)$ is generated by $(\pi_r)_*(x_{r,1}^k)$ for $k = 0,\cdots,r$ as a $\K_0(\Pro^6,T_2)$-module.

Using the localization method from the previous Section \ref{subsec.pushforward} on pushforwards along multiplication maps and combining with the relation $t = (ab)^{2}$, we have the following results.

For $r = 1$,
\begin{align*}
(\pi_1)_*(x_{1,1}^0) &= -\left(a^{10} b^{2} + a^{9} b^{3} + a^{8} b^{4} + a^{7} b^{5} + a^{6} b^{6} + a^{5} b^{7} + a^{4} b^{8} + a^{3} b^{9} + a^{2} b^{10}\right) t^{-2} \\
&\phantom{{} = }+ \left(a^{5} b + 2 a^{4} b^{2} + 2 a^{3} b^{3} + 2 a^{2} b^{4} + a b^{5}\right) t^{-1} \\
&\phantom{{} = }+ 1, \\
(\pi_1)_*(x_{1,1}^1) &= -\left(a^{10} b^3 + a^9 b^4 + a^8 b^5 + a^7 b^6 + a^6 b^7 + a^5 b^8 + a^4 b^9 + a^3 b^{10}\right) t^{-2} \\
&\phantom{{} = }+ \left(a^5 b^2 + 2a^4 b^3 + 2a^3 b^4 + a^2 b^5\right) t^{-1} \\
&\phantom{{} = }+ (a + b).
\end{align*}

We checked using SymPy \cite{sympy} that all other pushforwards are contained in the ideal generated by $(\pi_1)_*(1)$ and $(\pi_1)_*(x_{1,1})$.
In addition, the relation in $\K_0(\Pro^6,T_2)$ is also contained
in this ideal. In terms of symmetric polynomials in two variables, with notation
\[e_{-1}(a,b) = a^{-1} + b^{-1},\quad e_1(a,b) = a + b,\quad e_2(a,b) = ab,\]
corresponding to the generators of the representation ring $R(\GL_2)$ as the
$\Z_2$-invariant subring of $R(T_2)$, set
\begin{align*}
\alpha_{1,0}  &= 1 + e_{-1}h_3 - e_2^{-2}h_8 \\
&= h_0 + e_{-1}h_3 - e_2^{-2}h_8, \\
\alpha_{1,1} &= e_1 + e_1h_2 - e_2^{-1}h_7 \\
&= h_1 + e_1h_2 - e_2^{-1}h_7,
\end{align*}
where we denote by $h_n$ the complete homogeneous polynomial of degree $n$ in two variables $a,b$ corresponding to the representation $\Sym^n V$.
Note that $h_n$ is recursively given by $h_n - \epsilon h_{n - 1} + \lambda h_{n - 2} = 0$ with $h_0 = 1$ and $h_1 = \epsilon$.
Thus far we have shown that the image of
$\GG_0(\Aff(1,6)\setminus \Aff_{sm}(1,6), T_2)$ is generated
as an $R(T_2)$-module by the classes $\alpha_{1,0}, \alpha_{1,1} \in R(\GL_2)
\subset R(T_2)$.

Next we show that the
image of $\GG_0(\Aff(1,6)\setminus \Aff_{sm}(1,6), \GL_2)$ is generated by
the same generators as an $R(\GL_2)$-module. The argument is essentially
identical to the argument used in the proof of \cite[Theorem 4.9]{EdFu:09}.
Here we use the fact that the restriction map $R(\GL_2) \to R(T_2)$ is a
split injection of $R(\GL_2)$-modules so that $R(\GL_2)$ can
be identified with a summand in $R(T_2)$ \cite[Lemma 4.1]{EdGr:00}.
It follows that if $f \in R(\GL_2)$ is
in the image of $\GG_0(\Aff(1,6)\setminus \Aff_{sm}(1,6), \GL_2)$
we can write $f = p \alpha_{1,0} + q \alpha_{1,1}$ with $p, q \in R(T_2)$.
The elements $p,q \in R(T)$ can be written as $p = p_0 + p_1, q = q_0 + q_1$
with $p_0, q_0 \in R(\GL_2)$ and $p_1, q_1$ in a complementary $R(\GL_2)$-submodule $M$.
Thus we can write
$$ f= p_0 \alpha_{1,0} + q_0 \alpha_{1,1} + p_1 \alpha_{1,0} + q_1 \alpha_{1,1}.$$
But, $p_1 \alpha_{1,0} + q_1 \alpha_{1,1}$ lies in the complementary submodule
$M$, so it must be zero. Hence $f$ is in the ideal of $R(\GL_2)$ generated
by $\alpha_{1,0}$ and $\alpha_{1,1}$.

By \cites{Vis:98,EdFu:09}, with the identification $\M_2 \cong [\Aff_{sm}(1,6)/\GL_2]$, the Hodge bundle $\E$ on $\M_2$ corresponds to the defining representation $V$ of $\GL_2$. Therefore, we express $\epsilon = [\E] = e_1$ and $\lambda = [\det \E] = e_2$ and thus we proved Theorem \ref{thm-M2}. The 18 generators for $\K_0(\M_2)$ as an abelian group are given in Appendix~\ref{app.Zbasis}, and were computed using the $\tt{basis}$ command in Macaulay2. We then checked that these generators also form a basis for $\K_0(\M_2)_{\Q}$ (again using
the $\tt{basis}$ command with ground field $\Q$) which implies that $\K_0(\M_2)$ is a free abelian group
of rank 18.

\section{The Grothendieck rings of strata $\Delta_1$ and $\Mbar_2\setminus\Delta_1$}

To compute $\K_0(\Mbar_2)$ we stratify $\Mbar_2$ into the boundary
divisor $\Delta_1$, parametrizing curves with a separating node, and
its complement $\Mbar_2 \setminus \Delta_1$. We use their descriptions
as global quotient stacks given by Larson in \cite{Lar:21} to compute their
Grothendieck rings separately.

\subsection{The $\K$-theory of $\Delta_1$}
Let us start with a brief review on the substack $\Delta_1$.
Let $G = T_2 \rtimes S_2$ be the normalizer of the maximal torus $T_2$ in $\GL_2$.
If $L_r$ is the character of $\G_m$ with weight $r$ then we denote by
$W_{a_1, \ldots , a_n} = L_{a_1, \ldots , a_n} \boxplus L_{a_1, \ldots , a_n}$ the $(2n)$-dimensional
representation of $G$ defined as follow. As a vector space
it is the direct sum $(L_{a_1} \oplus \ldots \oplus L_{a_n}) \oplus (L_{a_1} \oplus \ldots \oplus L_{a_n})$
where the first copy of $\G_m$ in $T_2 = \G_m \times \G_m$ acts with weights
$(a_1, \ldots , a_n)$ on the first
factor and the second copy of $\G_m$ acts with weights $(a_1, \ldots , a_n)$ on the second factor.
Finally the $S_2$-action permutes the summands.

From \cite[Section 3.1]{Lar:21}, we know
that
\[\Delta_1 \cong \Sym^2 \Mbar_{1,1} \cong (W_{4,6} \setminus (L_{4,6} \times 0 \cup 0 \times L_{4,6})) / G.\]

\subsubsection{The representation ring of $G = N_{\GL_2}(T_2)$}

Our first task is to find the Grothendieck ring of the classifying space $BG$ where $G = T_2 \rtimes S_2$. Along with the computation of $\K_0(BG)$ we also give an explanation of relations in the ring $\K_0(BG)$ in terms of its representation theory.

Let $V$ be the defining representation of $\GL_2$. By \cite[Section
  5]{Lar:21}, there is a natural embedding $G \hookrightarrow \GL_2$
by viewing $G$ as the stabilizer of an unordered pair of distinct
lines in $V$, and hence the classifying space $BG$ is identified with
the complement of the image of the Veronese map $\pi_1 : [\Pro
V_1/\GL_2] \to [\Pro V_2/\GL_2]$. By the localization sequence on
$\pi_1(\Pro^1) \subset \Pro^2$ together with the fact
\[\K_0(\Pro^2,\GL_2) = \dfrac{\Z[e_1,e_2^{\pm 1}][t]}{((1 - e_2t^{-1})[1 - (e_1^2 - 2e_2)t^{-1} + e_2^2t^{-2}])},\]
from Section \ref{subsec.pushforward}, it suffices to find the ideal generated by the pushforwards $(\pi_1)_*(1)$ and $(\pi_1)_*(x)$ in $\K_0(\Pro^2,\GL_2)$ where $x = [\O_{\Pro^1}(1)] \in \K_0(\Pro^1,\GL_2)$. It has been computed in Example \ref{eg.Veronese} that 
\begin{align*}
& (\pi_1)_*(1) = 1 - e_2^2t^{-2}, \\
& (\pi_1)_*(x) = e_1(1 - e_2t^{-1}).
\end{align*}
Observe that the relation in $\K_0(\Pro^2,\GL_2)$ can be rewritten as
\begin{align*}
&\phantom{{} = } (1 - e_2t^{-1})[1 - (e_1^2 - 2e_2)t^{-1} + e_2^2t^{-2}] \\ 
&= -e_1^2(1 - e_2t^{-1}) + (1 - e_2t^{-1})(1 + e_2t^{-1})^2,
\end{align*}
which implies it is contained in the ideal generated by $(\pi_1)_*(1)$ and $(\pi_1)_*(x)$. Therefore
\[\K_0(BG) = \dfrac{\Z[e_1,e_2^{\pm 1},t^{\pm 1}]}{(1 - e_2^2t^{-2},e_1(1 - e_2t^{-1}))}.\]

Now we describe $\K_0(BG)$ in terms of representations of $G$. Recall that $V$ is the defining representation of $\GL_2$. Its restriction to $G$ is the representation $W_1 = L_1 \boxplus L_1$ and
we denote its class by $\epsilon$.
It is the restriction of the class we denoted
by $e_1 \in R(\GL_2)$. Denote by $\lambda = [\det W_1]$ which is the restriction of $e_2 \in R(\GL_2)$.

\begin{thm}
The Grothendieck ring, \ie the representation ring, of $BG$ is given by
\[\K_0(BG) = \dfrac{\Z[\epsilon,\lambda^{\pm 1},\gamma^{\pm 1}]}{(1 - \gamma^2, \epsilon(1 - \gamma))},\]
where $\gamma$ is the pullback of the sign representation under the quotient $G \to S_2$.
\end{thm}
\begin{proof}
  Since $\gamma$ is the pullback of the sign representation of $S_2$ it is a line element that
  satisfies
  the relation $1 - \gamma^2 = 0$ in $\K_0(BG)$. The only such line element in our description
  of $\K_0(BG)$ is $e_2t^{-1}$. By substituting $\gamma = e_2t^{-1}$ we get the final ring structure of $\K_0(BG)$. 
\end{proof}

\subsubsection{Tensor relations among the irreducible representations of $G=N_{\GL_2}(T_2)$}
In this section, we describe some tensor relations
among the irreducible representations of $G= N_{\GL_2}(T_2)$
that we will need for our computation of $\K_0(\Delta_1)$. 
The quotient $G/T_2 = S_2$ is the Weyl group of $\GL_2$
and it acts on $T_2$ by permuting the torus factors. Let $\Gamma$
be the pullback of the sign representation of $S_2$.
Since $T_2 \subset G$ is a normal subgroup of finite index, we know from the Clifford's theory that all the irreducible representations of $G$ can be constructed using induced representations of characters of the maximal torus $T_2$.
The (left) $T_2$-coset decomposition of $G$ is 
\[G = T_2 \cup \left\{\left.\begin{pmatrix}
0 & t_1 \\ t_2 & 0
\end{pmatrix} \right| t_1,t_2 \in \G_m \right\}.\]
Choose representatives 
\[e = \begin{pmatrix}
1 & 0 \\ 0 & 1
\end{pmatrix},\quad \sigma = \begin{pmatrix}
0 & 1 \\ 1 & 0
\end{pmatrix},\]
for these cosets.
If $W$ is any character of $T_2$, then the induced representation $V=\Ind_{T_2}^GW$ decomposes
as $V = W^e \oplus W^\sigma$ corresponding to the two cosets of $T_2$ in $G$.
Following the notation of \cite{FuHa:91} let $w$ be a
basis element for $W=W^e$ so that  $V$ is spanned by the elements
$w \in W^e, \sigma w \in W^\sigma$.
For simplicity, let 
\[t = \begin{pmatrix}
t_1 & 0 \\ 0 & t_2
\end{pmatrix},\quad t' = \begin{pmatrix}
0 & t_1 \\ t_2 & 0
\end{pmatrix}.\]

If $W$ is a character of $T_2$ then
\begin{align*}
& t \cdot \sigma w = \sigma \left(\begin{pmatrix} t_2 & 0\\
0 & t_1 \end{pmatrix} \cdot w\right), \\
& t' \cdot w = \sigma\left(\begin{pmatrix} t_2 & 0\\
0 & t_1 \end{pmatrix}\cdot w\right), \\
& t' \cdot \sigma w = t \cdot w.
\end{align*}
  
Using this description of the induced representation of a character
of $T_2$ we obtain the following identities.

\begin{itemize}
\item[(1)] If $1$ denotes the trivial representation
  then $\Ind_{T_2}^G(1) = 1 \oplus \Gamma$ where
  the trivial summand has basis $w + \sigma w$ and the summand
  corresponding to $\Gamma$ has basis $w - \sigma w$.

\item[(2)] $\Ind_{T_2}^G(\chi_1^m) = W_m$.

\item[(3)] $\Ind_{T_2}^G(\chi_2^m) = W_m \otimes \Gamma$.

\item[(4)] $\Ind_{T_2}^G(\chi_1^m \chi_2^m) = \det W_m \otimes (1 \oplus \Gamma)$.
  
\end{itemize}

\begin{prop} \label{prop.recur}
There is a relation among representations of $G$
\[W_{n - 1} \otimes W_1 = W_n \oplus (W_{n - 2} \otimes \wedge^2 W_1 \otimes \Gamma).\]
\end{prop}
\begin{proof}
It suffices to show that both sides are induced representations from the same representation of $T_2$. Consider the following representation $W$ of $T_2$
\begin{align*}
W &= \chi_1^n \oplus (\chi_2^{n - 1} \otimes \chi_1) \\
&= (\chi_1^{n - 1} \otimes \chi_1) \oplus (\chi_2^{n - 1} \otimes \chi_1) \\
&= (\chi_1^{n - 1} \oplus \chi_2^{n - 1}) \otimes \chi_1.
\end{align*}
Therefore we get 
\begin{align*}
\Ind_{T_2}^G(W) &= \Ind(\chi_1^n \oplus (\chi_2^{n - 1} \otimes \chi_1)) \\
&= \Ind(\chi_1^n) \oplus \Ind(\Res(\wedge^2 W_1) \otimes \chi_2^{n - 2}) \\
&= W_n \oplus (\wedge^2 W_1 \otimes W_{n - 2} \otimes \Gamma),
\end{align*}
and
\begin{align*}
\Ind_{T_2}^G(W) &= \Ind((\chi_1^{n - 1} \oplus \chi_2^{n - 1}) \otimes \chi_1) \\
&= \Ind(\Res(W_{n - 1}) \otimes \chi_1) \\
&= W_{n - 1} \otimes W_1
\end{align*}
which give the desired result.
\end{proof}

\begin{prop} \label{prop.duals}
The following relations hold in $R(G)$.
\begin{itemize}
\item[(1)] $[W_n^\vee] = [W_{-n}]$ and $[\wedge^2 W_n^\vee] = [\wedge^2 W_{-n}]$; 
\item[(2)] $[W_1] = [W_{-1}][\wedge^2 W_1]$.
\item[(3)] $[\Gamma \otimes W_n] = [W_n]$.
\end{itemize}

\end{prop}
\begin{proof}
It can be easily checked directly that $[W_n^\vee] = [W_{-n}]$ and $[\wedge^2 W_1^\vee] = [\wedge^2 W_{-1}]$. Moreover by setting $n = 1$, we have
\[W_{0} \otimes W_1 = W_1 \oplus (W_{-1} \otimes \wedge^2 W_{1} \otimes \Gamma)\] 
which leads to a relation $[W_1] = [W_{-1}][\wedge^2 W_1]$. For (3), first note that it can be derived from the relation $\epsilon(1 - \gamma) = 0$. A direct proof also shows
\begin{align*}
\Ind(\chi_1^m\otimes\chi_2^n) &= \Ind(\Res(\wedge^2 W_m) \otimes \chi_2^{n - m}) = \wedge^2 W_m \otimes \Gamma \otimes W_{n - m} \\
&= \Ind(\Res(\wedge^2 W_n) \otimes \chi_1^{m - n}) = \wedge^2 W_n \otimes W_{m - n}.
\end{align*}
Together with $[W_{-n}] = \lambda^{-n}[W_n]$ we obtain $[\Gamma \otimes W_n] = [W_n]$ as desired.
\end{proof}

By a direct computation, we get 
\[[\wedge^2 W_n] = \begin{cases}
\gamma\lambda^n &\text{if $n$ is even}; \\
\lambda^n &\text{if $n$ is odd},
\end{cases}\]
and likewise $[\wedge^2 W_n^\vee] = \gamma\lambda^{-n}$ if $n$ is even and $[\wedge^2 W_n^\vee] = \lambda^{-n}$ if $n$ is odd. 
In conclusion,
\begin{lem}\label{lem.pushforward}
The pushforward map along $\pi : R(T_2) \to R(G)$ can be described as follows
\begin{align*}
& \pi_*(1) = [1 \oplus \Gamma], \\
& \pi_*(a^mb^n) = [\wedge^2 W_n \otimes W_{m - n}] = [\wedge^2 W_m \otimes W_{n - m} \otimes \Gamma], \text{ for $m \neq n$},\\
& \pi_*(a^mb^m) = [\wedge^2 W_m \otimes W_0],
\end{align*}
where, as before,  $a$ and $b$ denote the $\K$-theoretic classes of
$\chi_1$ and $\chi_2$ respectively.
\end{lem}

\subsection{Computation of $\K_0(\Delta_1)$}

Recall that $\Delta_1 \cong (W_{4,6} \setminus (L_{4,6} \times 0 \cup 0 \times L_{4,6}))/G$. We first excise $0 \times 0$ from $W_{4,6}$ and by the localization sequence 
\[\K_0(W_{4,6}\setminus (0 \times 0),G) = \dfrac{\K_0(BG)}{(\lambda_{-1}(W_{4,6}^\vee))}.\]

Expanding 
\begin{align*}
\lambda_{-1}(W_{4,6}^\vee) &= \lambda_{-1}(W_4^\vee)\lambda_{-1}(W_6^\vee) \\
&= (1 - [W_4^\vee] + [\wedge^2 W_4^\vee])(1 - [W_6^\vee] + [\wedge^2 W_6^\vee]),
\end{align*}
using the identities of Proposition \ref{prop.duals} together with
the recursive formula of Proposition \ref{prop.recur} we obtain
that
\begin{align*}
[W_4^\vee] &= \lambda^{-4}[\epsilon^4 - 4\epsilon^2\lambda + (1 + \gamma)\lambda^2], \\
[W_6^\vee] &= \lambda^{-6}[\epsilon^6 - 6\epsilon^4\lambda + 9\epsilon^2\lambda^2 - (1 + \gamma)\lambda^3],
\end{align*}
which yield a formula for $\lambda_{-1}(W_{4,6}^{\vee})$ in terms
of the generators $\epsilon, \lambda^{\pm 1}, \gamma$ of $R(G)$.

In order to compute $\K_0(\Delta_1)$, we need to further excise
$(((L_{4,6} \setminus 0) \times 0) \cup (0 \times (L_{4,6} \setminus
0)))/G$ in $(W_{4,6} \setminus 0 \times 0)/G$. Using the same argument
in \cite[Section 8]{Lar:21}, it suffices to find the images of $1,a$
which are pushforwards of $1 \cdot \lambda_{-1}(N^\vee_{L_{4,6} \times
  0} (L_{4,6} \oplus L_{4,6}))$ and $a\cdot
\lambda_{-1}(N^\vee_{L_{4,6} \times 0} (L_{4,6} \oplus L_{4,6}))$
along the map $R(T_2) \to R(G)$ as described in Lemma
\ref{lem.pushforward} by taking induced representations.

Notice that $N^\vee_{L_{4,6} \times 0} (L_{4,6}\oplus L_{4,6}) = b^{-4} + b^{-6}$ and thus $\lambda_{-1}(N^\vee_{L_{4,6} \times 0} (L_{4,6}\oplus L_{4,6})) = (1 - b^{-4})(1 - b^{-6})$. Therefore we can write down explicitly the pushforwards of $(1 - b^{-4})(1 - b^{-6})$ and $a(1 - b^{-4})(1 - b^{-6})$ as follows
\begin{align*}
R_1 &\coloneq (1 + \gamma) - \gamma(\lambda^{-4}[W_4] + \lambda^{-6}[W_6] - \lambda^{-10}[W_{10}]), \\
R_2 &\coloneq \epsilon - \gamma(\lambda^{-4}[W_5] + \lambda^{-6}[W_7] - \lambda^{-10}[W_{11}]).
\end{align*}
Once again these can be expressed in terms of the generators $\epsilon, \lambda^{\pm 1}, \gamma$ for $R(G$).

\begin{thm}\label{thm.delta_1}
The Grothendieck ring of $\Delta_1$ is given by 
\[\K_0(\Delta_1) = \Z[\epsilon,\lambda^{\pm 1},\gamma^{\pm 1}]/(1 - \gamma^2,\epsilon(1 - \gamma),R_0,R_1,R_2)\]
where $R_0 = \lambda_{-1}(W_{4,6}^\vee)$ and $R_1,R_2$ are defined as above.
\end{thm}

\begin{cor}\label{cor.tor.delta_1}
The Grothendieck ring $\K_0(\Delta_1)$ is a free abelian group of rank $65$.
\end{cor}
\begin{proof}
Once again using $\tt{basis}$ command in Macaulay2 \cite{M2} 
we checked that $\K_0(\Delta_1)$ is generated as an abelian group
by the $65$ generators listed in Appendix~\ref{app.Zbasis}.
Again we checked that these generators form a basis for  the vector space
$\K_0(\Delta_1) \otimes_\Z \Q$ which implies that
$\K_0(\Delta_1)$ is free abelian of rank 65.
Note that, the
same rank result can also be obtained via the Riemann-Roch isomorphism
by computing the rational Chow group 
$\Chow^*(I\Delta_1)_\Q$ of the inertia stack $I \Delta_1$.
\end{proof}

The Hodge bundle $\E$ restricted to $\Delta_1$ has fiber
over $[C] \in \Delta$ where $C = E_1 \cup_p E_2$ given by
\[H^0(\omega_{E_1}) \oplus H^0(\omega_{E_2}),\]
with $H^0(\omega_E) = \langle dx/y \rangle$.
Thus we can identify the pullback of $\E$ to $\K_0(\Delta_1)$ with
the pullback of the bundle $W_1$ in $\K_0(BG)$ \cite[cf. Section 3.1]{Lar:21},
which is indeed consistent with the notation used for class $\lambda$.
The following is a $\K$-theoretic analogue of  \cite[Lemma 6.1]{Lar:21}. 

\begin{lem}\label{lemma.nonzerodivisor}
  We have $[\O_{\Delta_1}(\Delta_1)] = \gamma\lambda^{-1}$ in $\K_0(\Delta_1)$.
\end{lem}
\begin{proof}
  The normal bundle $N_{\Delta_1}\Mbar_2 = \O_{\Delta_1}(\Delta_1)$ has fiber over
  a point $[C] \in \Delta_1$ given by
  $H^0(\omega_{E_1})^\vee \otimes H^0(\omega_{E_2})^\vee$. Since $\wedge^2 W_1 = \det \E \otimes \Gamma$, we have that
\[N_{\Delta_1}\Mbar_2 = \wedge^2 W_1^\vee \otimes \Gamma\]
as $\Gamma$ is self-dual in $\K_0(\Delta_1)$. Therefore we get
\[[\O_{\Delta_1}(\Delta_1)] = \gamma\lambda^{-1}. \qedhere\]
\end{proof}
Using Lemma \ref{lemma.nonzerodivisor} we can conclude that
$\K_0(\Delta_1)$ is generated by classes pulled back from $\Mbar_2$. Precisely we have the following presentation for $\K_0(\Delta_1)$.
\begin{cor} \label{cor.delta_1}
  The Grothendieck ring of $\Delta_1$ has the following presentation
  $$ \K_0(\Delta_1) = \Z[\epsilon, \lambda^{\pm 1}, \delta_1^{\pm 1}]/I_{\Delta_1}.$$
  Here $\epsilon$ is the pullback of the Hodge bundle,
  and $\delta_1$ is the pullback of the line bundle $\delta_1$ from  $\Mbar_2$
and
$$I_{\Delta_1} = \left(1 -(\lambda \delta_1)^2, \epsilon(1 - \lambda \delta_1),
R_0(\epsilon, \lambda, \lambda \delta_1), R_1(\epsilon, \lambda, \lambda \delta_1), R_2(\epsilon, \lambda, \lambda \delta_1)\right).$$
\end{cor}
\begin{remark}
  Since $\K_0(\Delta_1)$ is generated by pullbacks of classes in $\K_0(\Mbar_2)$
  the projection formula implies that the pushforward $i_* \colon \K_0(\Delta_1) \to \K_0(\Mbar_2)$ is multiplication by $i_*([{\mathcal O}_{\Delta_1}]) =
  1 -\delta_1^{-1}$ where the equality comes from the short exact
  sequence
  $$ 0 \to {\mathcal O}_{\Mbar_2}(-\Delta_1) \to {\mathcal O}_{\Mbar_2} \to {\mathcal O}_{\Delta_1} \to 0.$$
\end{remark}

\subsection{Computation of $\K_0(\Mbar_2 \setminus \Delta_1$)}
As observed in \cite[Section 3]{Lar:21} the stack
$\Mbar_2 \setminus \Delta_1$ can be identified with the
stack $[(V_6 \setminus \Theta)/ \GL_2]$.
were $V_6$ is the vector space of binary sextics and $\Theta$ is the locus of sextics with triple roots. Here $\GL_2$ acts on $V_6$ by 
  $A \cdot f(x) = \det(A)^2 f(A^{-1}x)$.
We can compute $\K_0(\Mbar_2 \setminus \Delta_1)$
using the same technique we used to compute $\K_0(\M_2)$ by passing to the projectivization and restricting to the action of the maximal torus.

\begin{lem}\label{lem.compl.surj}
Let $t$ denote the hyperplane class $[\O_{\Pro^6}(1)]$ on $\Pro^6$. Then 
\[\K_0(\Pro^6 \setminus \Theta,\GL_2) \to \K_0(\Mbar_2 \setminus \Delta_1)\]
is surjective with kernel given by $1 - te_2^{-2} = 0$.
\end{lem}

As in Section \ref{sec.chowenv} we can use
an equivariant Chow envelope to compute the image of $\GG_0(\Theta,T_2)$ in
$\K_0(\Pro^6, T_2)$. In this case a $\GL_2$ and hence $T_2$-equivariant
Chow envelope for $\Theta$ is
given by 
\[\tau : \coprod_{r = 1}^2 \Pro^r \times \Pro^{6 - 3r} \to \Theta.\]
We denote by $\tau_r$ the restriction of $\tau$ to $\Pro^r \times \Pro^{6 - 3r}$.
Using the $T_2$-equivariant localization technique of Section \ref{subsec.pushforward}
with the cube map replacing the square map we find using SymPy \cite{sympy} that the 
ideal generated by the image of the pushforward
\[\K_0(\Theta,T_2) \to \K_0(\Pro^6,T_2)\]
is generated
by 
\begin{align*}
& (\tau_1)_*(x_{1,1}^0) = t^{-3} e_2^{4} h_{10} - t^{-2} e_2^{2} h_2h_6 + t^{-1} e_2 h_2^2 + 1, \\
& (\tau_1)_*(x_{1,1}^1) = t^{-3}e_2^5 h_9 - t^{-2} e_2^3 h_2h_5 + t^{-1} e_2^2 e_1h_2 + e_1,
\end{align*}
and
\[(\tau_2)_*(x_{2,1}^0) = t^{-6} e_2^{18} + t^{-5} e_2^{14} h_2 - t^{-4} e_2^8 h_2h_6 + t^{-3} e_2^4 (e_1h_9 + e_2^3h_4 + e_2^3h_2^2) - t^{-2} e_2^2 h_2h_6 + t^{-1} e_2^2 h_2 + 1,\]
where $t$ denotes the class of $\O(1)$ in $\K_0(\Pro^6,T_2)$,
and $h_n$ denotes the the complete homogeneous polynomial of degree $n$
in terms of $a,b$, the classes of characters $\chi_1,\chi_2$ of the torus $T_2$,
together with $e_1 = a + b$ and $e_2 = ab$.
Note that $h_n$ can be recursively expressed by
the elementary symmetric polynomials $e_1,e_2$.
In addition the relation in $\K_0(\Pro^6,T_2)$ is also contained in this ideal.
Combining with the relation $t = e_2^2$ from Lemma \ref{lem.compl.surj}, we set
\begin{equation} \label{eq.esses}
\begin{split}
& S_{10} \coloneq e_2^{-2}h_{10} - e_2^{-2}h_2h_6 + e_2^{-1}h_2^2 + 1, \\
& S_{11} \coloneq e_2^{-1}h_9 - e_2^{-1}h_2h_5 + e_1h_2 + e_1, \\
& S_{20} \coloneq e_2^6 + e_2^4h_2 - h_2h_6 + e_2^{-2}(e_1h_9 + e_2^3h_4 + e_2^3h_2^2) - e_2^{-2}h_2h_6 + h_2 + 1.
\end{split}
\end{equation}
Using an argument similar to the one in Section \ref{sec.chowenv},
we obtain the following presentation for the Grothendieck ring $\K_0(\Mbar_2\setminus \Delta_1)$.

\begin{thm}\label{thm.complidelta_1}
The Grothendieck ring of the complement $\Mbar_2 \setminus \Delta_1$ is given by
\[\K_0(\Mbar_2\setminus \Delta_1) = \Z[e_1,e_2^{\pm 1}]/(S_{10},S_{11},S_{20})\]
where the relations $S_{10},S_{11}$ and $S_{20}$ are defined as above.
\end{thm}

\begin{cor}\label{cor.tor.complidelta_1}
The Grothendieck ring $\K_0(\Mbar_2 \setminus \Delta_1)$ is a free abelian group of rank $32$.
\end{cor}
\begin{proof}
  The proof is similar to that of Corollary \ref{cor.tor.delta_1} based on the ring structure obtained in Theorem \ref{thm.complidelta_1}.
  A basis for $\K_0(\Mbar_2\setminus\Delta_1)$ is given in Appendix~\ref{app.Zbasis}.
\end{proof}


\section{Proof of Theorem \ref{thm-Mbar2}. The Grothendieck ring of the moduli stack $\Mbar_2$} 
In this section, we conclude with a presentation for the Grothendieck ring $\K_0(\Mbar_2)$.
\subsection{Description of $\K_0(\Mbar_2)$ as an  abelian group}
\begin{prop}
  The Grothendieck ring $\K_0(\Mbar_2)$ is a free abelian group of rank 97 which fits into a
  short exact sequence
\begin{align}\label{diag.ses}
0 \to \K_0(\Delta_1) \stackrel{i_*}{\longrightarrow} \K_0(\Mbar_2) \stackrel{j^*}{\longrightarrow} \K_0(\Mbar_2\setminus\Delta_1) \to 0.
\end{align}
\end{prop}
\begin{proof}
Consider the localization exact sequence
\begin{align*}
\K_0(\Delta_1) \stackrel{i_*}{\longrightarrow} \K_0(\Mbar_2) \stackrel{j^*}{\longrightarrow} \K_0(\Mbar_2\setminus\Delta_1) \to 0.
\end{align*}
Since $\K_0(\Delta_1)$ and $\K_0(\Mbar_2 \setminus \Delta_1)$ are free abelian
of ranks $65$ and $32$ respectively, we know that $\K_0(\Mbar_2)$ has
rank at most $97$. To prove the proposition it suffices
to prove that $\K_0(\Mbar_2)_\Q$ has rank exactly $97$, 
since it implies that $\ker i_*$ must be trivial 
as $\K_0(\Delta_1)$ is free abelian.

By the Riemann-Roch isomorphism for Deligne-Mumford stacks (Theorem \ref{thm.rr})
we know that $\K_0(\Mbar_2)_\Q$ is isomorphic to $\Chow^*(I\Mbar_2)_\Q$
where $I\Mbar_2$ is the inertia stack of $\Mbar_2$.
Thus, we are reduced to computing the rank of $\Chow^*(I\Mbar_2)_\Q$.
This can be done directly but since Pagani \cite{Pag:12, Pag:09} has previously
computed that $H^{2*}(I\Mbar_2)_\Q$ has rank 97, our proposition
follows from the following lemma.
\end{proof}


\begin{lem}\label{prop.RR.Mbar2}
The cycle map
\[\Chow^*(I\Mbar_2)_\Q \longrightarrow H^{2*}(I\Mbar_2)_\Q\]
is an isomorphism.
\end{lem}
\begin{proof}
  Pagani's analysis of the inertia $I\Mbar_2$ \cite{Pag:12}
  shows that it is a disjoint
  sum of the 
following sectors:
  \begin{enumerate}
  \item The full sector $\Mbar_2$. On this sector we know that $\Chow^*(\Mbar_2)_\Q = H^{2*}(\Mbar_2)_\Q$ and is a vector space of rank six, so the result
    holds for the full sector.
  \item Stacky points. For such sectors the Chow and cohomology groups have rank one.
  \item Stacky $\Pro^1$'s. On these sectors both the Chow and cohomology groups
    have rank $2$.
  \item A unique two-dimensional sector whose coarse moduli space
    is isomorphic to $\Mbar_{0,5}/S_3$. 
Since $S_3$ is a finite group and $\Chow^*(\Mbar_{0,5}) =  H^{2*}(\Mbar_{0,5})$,
it implies $\Chow^*(\Mbar_{0,5}/S_3)_\Q = H^{2*}(\Mbar_{0,5}/S_3)_\Q$, as both
groups can be identified as the $S_3$-invariants in the Chow or cohomology
group of $\Mbar_{0,5}$. \qedhere
  \end{enumerate}
 \end{proof}

\subsection{A presentation for the ring $\K_0(\Mbar_2)$}
From the short exact sequence \eqref{diag.ses},
the ring $\K_0(\Mbar_2)$ is generated by classes which pullback to generators
for $\K_0(\Mbar_2\setminus\Delta_1)$
together with the pushforwards of the generators for $\K_0(\Delta_1)$.
We know that $\K_0(\Mbar_2\setminus\Delta_1)$, is generated by $\{e_1,e_2^{\pm 1}\}$ which are just the pullbacks
of the Hodge bundle $\epsilon$ and $\lambda^{\pm 1}$ respectively.
Also any class in the image of $\K_0(\Delta_1)$ can be expressed
as $(1-\delta_1^{-1})p(\epsilon, \lambda^{\pm 1}, \delta_1^{\pm 1})$.
Therefore, any element in $\K_0(\Mbar_2)$ can be expressed as polynomial in
$\epsilon, \lambda^{\pm 1}, \delta_1^{\pm 1}$.
Let $A = \Z[\epsilon,\lambda^{\pm 1},\delta_1^{\pm 1}]$,
and consider the surjective map $\pi : A \to \K_0(\Mbar_2)$, 
which fits into the following diagram of $A$-modules.

\[\begin{tikzcd}
0 \ar[r]  & \ker\pi_1 \ar[r,"i_*",shift left = 1.8pt] \ar[d,hook] & \ker\pi \ar[l,"i^*",shift left = 1.8pt]\ar[r,"j^*"]\ar[d,hook] & \ker \pi_2 \ar[d,hook] & \\
0 \ar[r] & A \ar[r,"i_*",shift left = 1.8pt] \ar[d,"\pi_1"] & A \ar[l,"i^*",shift left = 1.8pt]\ar[d,"\pi"] \ar[r,"j^*"] & A/\left(1 - \delta_1^{-1}\right) \ar[r]\ar[d,"\pi_2"] & 0 \\
0 \ar[r] & \K_0(\Delta_1) \ar[d] \ar[r,"i_*",shift left = 1.8pt] & \K_0(\Mbar_2) \ar[d]\ar[r,"j^*"]\ar[l,"i^*",shift left = 1.8pt] & \K_0(\Mbar_2\setminus\Delta_1) \ar[r]\ar[d] & 0 \\
 & 0 & 0 & 0 &
\end{tikzcd}\]
Note the map $i_* \colon A \to A$ is multiplication by $(1-\delta_1^{-1})$
and the pullback $i^* \colon A \to A$ is the identity.
Also note that the projection $j^* \colon A \to A/(1-\delta_1^{-1}) = \Z[\epsilon, \lambda^{\pm 1}]$
has a section given by the inclusion $s \colon \Z[\epsilon, \lambda^{\pm 1}]
  \hookrightarrow \Z[\epsilon, \lambda^{\pm 1}, \delta_1^{\pm 1}]$
  and we let $I_{\Mbar_2 \setminus \Delta_1}$ be the ideal in $A$ generated
  by the image of $\ker \pi_2$. This ideal is generated by
  the polynomials $S_{10}, S_{11}, S_{20}$ defined in \eqref{eq.esses}
  viewed as polynomials in $A$.

\begin{prop}\label{prop.kernel}
The relations in the Grothendieck ring $\K_0(\Mbar_2)$,
\ie the kernel $\ker\pi$, is the ideal
$(I_{\Delta_1} \cap I_{\Mbar_2\setminus\Delta_1}) + \left(1 - \delta_1^{-1}\right)I_{\Delta_1}$. 
\end{prop}
\begin{proof}
By the snake lemma, we have the exact sequences on kernels
\[0 \to \ker \pi_1 \stackrel{i_*} \to \ker\pi \stackrel{j^*}\to \ker \pi_2 \to 0.\]
In addition the section $s \colon A/(1-\delta_1^{-1}) \to A$ induces
a section $s \colon \ker \pi_2 \to \ker \pi$. Thus
any element $x \in \ker \pi$ can be (uniquely) expressed
as $x = a + b$ with $a \in i_* I_{\Delta_1} = (1- \delta_1^{-1}) I_{\Delta_1}$
and $b \in s(\ker\pi_2) \subset I_{\Mbar_2 \setminus \Delta_1}$.
Moreover, 
$b \in I_{\Delta_1}$ as well because
$i^* \ker \pi \subseteq \ker \pi_1 = I_{\Delta_1}$. Hence
we conclude that
$\ker \pi \subseteq (I_{\Mbar_2 \setminus \Delta_1} \cap I_{\Delta_1}) + (1-\delta_1^{-1})I_{\Delta_1}$.

Since $A$ is a domain,
$1 - \delta_1^{-1}$ is not a zero-divisor in $A$,
and by a simple generalization of \cite[Lemma 4.4]{VeVi:03}, 
we obtain the following cartesian diagram of $A$-modules
\[\begin{tikzcd}
\ker\pi \ar[r,"j^*"] \ar[d,"i^*"] & \ker \pi_2 \ar[d] \\
\ker \pi_1 \ar[r] & \dfrac{\ker \pi_1}{\left(1 - \delta_1^{-1}\right)}
\end{tikzcd}\]
where $\ker\pi$ can be viewed as an ideal of 
ring $A$ which is identified with the fiber product 
\[\begin{tikzcd}
A \ar[r,"j^*"] \ar[d,"i^*"] & A/(1 - \delta_1^{-1}) \ar[d,"\overline{i^*}"] \\
A \ar[r] & A/\left(1 - \delta_1^{-1}\right)
\end{tikzcd}\]
by $a \mapsto (a,\overline{a})$ where $\overline{a}$ is the class of $a$
in the quotient ring $A/\left(1 - \delta_1^{-1}\right)$.

In particular, it follows from this description that $I_{\Delta_1}
\cap I_{\Mbar_2 \setminus \Delta_1} \subseteq \ker \pi$, since if
$x \in I_{\Delta_1} \cap I_{\Mbar_2 \setminus \Delta_1}$
  then its images under the two maps $\ker\pi \to \ker \pi_1/(1-\delta_1^{-1})$
  are equal.
  Hence we conclude that
  \[\ker \pi = (I_{\Delta_1} \cap I_{\Mbar_2 \setminus \Delta_1}) + (1-\delta_1^{-1})I_{\Delta_1}. \qedhere\]
\end{proof}  

\section*{Data Availability Statement}
No datasets were generated or analyzed during the current study.


\appendix
\section{Bases for $\K_0(\M_2), \K_0(\Delta_1)$ and $\K_0(\Mbar_2 \setminus \Delta_1)$ as abelian groups}\label{app.Zbasis}
A $\Z$-basis for $\K_0(\M_2)$ as finitely generated abelian group is given by
\[1, \epsilon, \epsilon^2, \epsilon^3, \epsilon^4, \epsilon^4\lambda^{-1}, \epsilon^3\lambda^{-1}, \epsilon^2\lambda^{-1}, \epsilon\lambda, \epsilon\lambda^2, \epsilon\lambda^{-1}, \epsilon\lambda^{-2}, \epsilon\lambda^{-3}, \lambda, \lambda^2, \lambda^{-1}, \lambda^{-2}, \lambda^{-3}.\]

We also provide $\Z$-bases for $\K_0(\Delta_1)$ and $\K_0(\Mbar_2 \setminus \Delta_1)$ separately as finitely generated abelian groups. For $\K_0(\Delta_1)$, it is generated as an abelian group by the following generators
\begin{align*}
& \epsilon^5, \epsilon^5 \delta, \epsilon^5 \delta^{-1}, \epsilon^5 \delta^{-2}, \epsilon^4, \epsilon^4 \delta, \epsilon^4 \delta^{-1}, \epsilon^4 \delta^2, \epsilon^4 \delta^{-2}, \epsilon^4 \delta^{-3}, \\
& \epsilon^3, \epsilon^3 \delta, \epsilon^3 \delta^{-1}, \epsilon^3 \delta^2, \epsilon^3 \delta^{-2}, \epsilon^3 \delta^3, \epsilon^3 \delta^{-3}, \epsilon^3 \delta^4, \\
& \epsilon^2, \epsilon^2 \delta, \epsilon^2 \delta^{-1}, \epsilon^2 \delta^2, \epsilon^2 \delta^{-2}, \epsilon^2 \delta^3, \epsilon^2 \delta^{-3}, \epsilon^2 \delta^4, \epsilon^2 \delta^{-4}, \epsilon^2 \delta^{-5}, \\
& \epsilon, \epsilon \delta, \epsilon \delta^{-1}, \epsilon \delta^2, \epsilon \delta^{-2}, \epsilon \delta^3, \epsilon \delta^{-3}, \epsilon \delta^4, \epsilon \delta^{-4}, \epsilon \delta^5, \epsilon \delta^{-5}, \epsilon \delta^6, \epsilon \delta^{-6}, \\
& \delta, \delta^2, \delta^3, \delta^4, \delta^5, \delta^6, \delta^{-1}, \delta^{-2}, \delta^{-3}, \delta^{-4}, \delta^{-5}, \delta^{-6}, \delta^{-7}, \\
& \delta \gamma, \delta^2 \gamma, \delta^3 \gamma, \delta^4 \gamma, \delta^{-1} \gamma, \delta^{-2} \gamma, \delta^{-3} \gamma, \delta^{-4} \gamma, \delta^{-5} \gamma, \gamma, 1.
\end{align*}

A $\Z$-basis for $\K_0(\Mbar_2 \setminus \Delta_1)$ is the following
\begin{align*}
& e_1^5, e_1^4, e_1^4 e_2^{-1}, e_1^3, e_1^3 e_2, e_1^3 e_2^{-1}, e_1^3 e_2^2, e_1^3 e_2^{-2}, e_1^2, e_1^2 e_2, e_1^2 e_2^{-1}, e_1^2 e_2^2, e_1^2 e_2^{-2}, e_1^2 e_2^{-3}, \\
& e_1, e_1 e_2, e_1 e_2^{-1}, e_1 e_2^2, e_1 e_2^{-2}, e_1 e_2^3, e_1 e_2^{-3}, e_1 e_2^4, e_1 e_2^{-4}, e_2, e_2^2, e_2^3, e_2^4, e_2^{-1}, e_2^{-2}, e_2^{-3}, e_2^{-4}, 1. 
\end{align*}


\begin{thebibliography}{10}

\bibitem{AnPa:15}
Dave Anderson and Sam Payne.
\newblock Operational {$K$}-theory.
\newblock {\em Doc. Math.}, 20:357--399, 2015.

\bibitem{Edi:13}
Dan Edidin.
\newblock Riemann-{R}och for {D}eligne-{M}umford stacks.
\newblock In {\em A celebration of algebraic geometry}, volume~18 of {\em Clay
  Math. Proc.}, pages 241--266. Amer. Math. Soc., Providence, RI, 2013.

\bibitem{EdFu:09}
Dan Edidin and Damiano Fulghesu.
\newblock The integral {C}how ring of the stack of hyperelliptic curves of even
  genus.
\newblock {\em Math. Res. Lett.}, 16(1):27--40, 2009.

\bibitem{EdGr:98}
Dan Edidin and William Graham.
\newblock Equivariant intersection theory.
\newblock {\em Invent. Math.}, 131(3):595--634, 1998.

\bibitem{EdGr:00}
Dan Edidin and William Graham.
\newblock Riemann-Roch for equivariant {C}how groups.
\newblock {\em Duke Math. J.}, 102(3):567--594, 2000.

\bibitem{EdGr:05}
Dan Edidin and William Graham.
\newblock Nonabelian localization in equivariant {$K$}-theory and
  {R}iemann-{R}och for quotients.
\newblock {\em Adv. Math.}, 198(2):547--582, 2005.


\bibitem{FuHa:91}
William Fulton and Joe Harris.
\newblock {\em Representation theory: a first course}, volume 129.
\newblock Springer Science \& Business Media, 2013.

\bibitem{M2}
Daniel~R. Grayson and Michael~E. Stillman.
\newblock Macaulay2, a software system for research in algebraic geometry.
\newblock Available at \url{https://math.uiuc.edu/Macaulay2/}.

\bibitem{Lar:21}
Eric Larson.
\newblock The integral {C}how ring of {$\overline M_2$}.
\newblock {\em Algebr. Geom.}, 8(3):286--318, 2021.

\bibitem{dilorenzo2019chow}
Andrea~Di Lorenzo.
\newblock The {C}how ring of the stack of hyperelliptic curves of odd genus.
\newblock {\em International Math Research Notices}, 2019.

\bibitem{dilorenzo2021stable}
Andrea~Di Lorenzo, Michele Pernice, and Angelo Vistoli.
\newblock Stable cuspidal curves and the integral {C}how ring of
  $\overline{\mathscr M}_{2,1}$.
\newblock {\em Geometry and Topology, to appear (arXiv:2108.03680).}, 2022.

\bibitem{ViDL:21}
Andrea~Di Lorenzo and Angelo Vistoli.
\newblock Polarized twisted conics and moduli of stable curves of genus two.
\newblock {\em arXiv:2103.13204}, 2021.

\bibitem{Mer:05}
Alexander~S. Merkurjev.
\newblock Equivariant {$K$}-theory.
\newblock In {\em Handbook of {$K$}-theory. {V}ol. 1, 2}, pages 925--954.
  Springer, Berlin, 2005.

\bibitem{sympy}
Aaron Meurer, Christopher~P. Smith, Mateusz Paprocki, Ond\v{r}ej
  \v{C}ert\'{i}k, Sergey~B. Kirpichev, Matthew Rocklin, AMiT Kumar, Sergiu
  Ivanov, Jason~K. Moore, Sartaj Singh, Thilina Rathnayake, Sean Vig, Brian~E.
  Granger, Richard~P. Muller, Francesco Bonazzi, Harsh Gupta, Shivam Vats,
  Fredrik Johansson, Fabian Pedregosa, Matthew~J. Curry, Andy~R. Terrel,
  \v{S}t\v{e}p\'{a}n Rou\v{c}ka, Ashutosh Saboo, Isuru Fernando, Sumith Kulal,
  Robert Cimrman, and Anthony Scopatz.
\newblock Sympy: symbolic computing in python.
\newblock {\em PeerJ Computer Science}, 3:e103, January 2017.

\bibitem{Mum:83}
David Mumford.
\newblock Towards an enumerative geometry of the moduli space of curves.
\newblock In {\em Arithmetic and geometry, {V}ol. {II}}, volume~36 of {\em
  Progr. Math.}, pages 271--328. Birkh\"{a}user Boston, Boston, MA, 1983.

\bibitem{Nie:74}
H.~Andreas Nielsen.
\newblock Diagonalizably linearized coherent sheaves.
\newblock {\em Bull. Soc. Math. France}, 102:85--97, 1974.

\bibitem{Pag:12}
Nicola Pagani.
\newblock The {C}hen-{R}uan cohomology of moduli of curves of genus 2 with
  marked points.
\newblock {\em Adv. Math.}, 229(3):1643--1687, 2012.

\bibitem{Pag:09}
Nicola~Tito Pagani.
\newblock {C}hen-{R}uan cohomology of moduli of curves.
\newblock 2009.

\bibitem{pernice2023almost}
Michele Pernice.
\newblock The (almost) integral {C}how ring of $\overline{\mathscr M}_3$.
\newblock {\em arXiv:2023:13614}, 2023.

\bibitem{Tho:87}
R.~W. Thomason.
\newblock Algebraic {$K$}-theory of group scheme actions.
\newblock In {\em Algebraic topology and algebraic {$K$}-theory ({P}rinceton,
  {N}.{J}., 1983)}, volume 113 of {\em Ann. of Math. Stud.}, pages 539--563.
  Princeton Univ. Press, Princeton, NJ, 1987.

\bibitem{VeVi:02}
Gabriele Vezzosi and Angelo Vistoli.
\newblock Higher algebraic {$K$}-theory of group actions with finite
  stabilizers.
\newblock {\em Duke Math. J.}, 113(1):1--55, 2002.

\bibitem{VeVi:03}
Gabriele Vezzosi and Angelo Vistoli.
\newblock Higher algebraic {$K$}-theory for actions of diagonalizable groups.
\newblock {\em Invent. Math.}, 153(1):1--44, 2003.

\bibitem{Vis:98}
Angelo Vistoli.
\newblock The {C}how ring of {$\mathscr M_2$}. {A}ppendix to ``{E}quivariant
  intersection theory'' [{I}nvent. {M}ath. {\bf 131} (1998), no. 3, 595--634;
  {MR}1614555 (99j:14003a)] by {D}. {E}didin and {W}. {G}raham.
\newblock {\em Invent. Math.}, 131(3):635--644, 1998.

\end{thebibliography}
\end{document}